\documentclass[11pt, reqno]{amsart}

\usepackage{fancyhdr}
\usepackage{psfrag}
\usepackage{graphicx}
\usepackage{amsmath}
\usepackage{amsfonts}
\usepackage{amssymb}		
\usepackage{amsthm}
\usepackage[indent]{caption}
\usepackage{indentfirst}
\usepackage{float}
\usepackage{latexsym}
\usepackage{graphics}
\usepackage{verbatim}
\usepackage[T1]{fontenc}
\usepackage{lmodern}
\usepackage{enumerate}
\usepackage{dsfont}

\usepackage{datetime}

\newcommand{\Q}{{\mathbb Q}}
\newcommand{\C}{{\mathbb C}}

\newcommand{\Kbar}{{\overline{K}}}

\newcommand{\Z}{{\mathbb Z}}

\newcommand{\N}{{\mathbb N}}
\newcommand{\FF}{{\mathbb F}}

\newcommand{\Gal}{{\rm Gal}}
\newcommand{\GL}{{\rm GL}}

\newcommand{\Aut}{{\rm Aut}}
\newcommand{\Li}{{\rm Li \; }}
\newcommand{\ord}{{\rm ord \; }}
\newcommand{\Frob}{{\rm Frob}}

\newcommand{\lambdanot}{{\lambda_0}}
\newcommand{\ellnot}{{\ell_0}}
\newcommand{\el}{{\ell}}

\newcommand{\TT}{{\mathcal T}}

\newcommand{\half}{\frac 12}

\newcommand{\notdiv}{{\not| ~~}}

\newtheorem{theorem}{Theorem}[section]
\newtheorem{lem}{Lemma}[section]
\newtheorem{prop}{Proposition}[section]
\newtheorem{cor}{Corollary}[section]
\newtheorem{defn}{Definition}[section]
\newtheorem{conjecture}{Conjecture}[section]
\newtheorem{remark}{Remark}[section]

\newcommand{\pf}{{\em Proof.}\hspace{10pt}}

\begin{document}

\input{bezier.sty}

\pagestyle{plain}
\pagenumbering{arabic}

\title{Tate Cycles on Abelian Varieties with Complex Multiplication} 
\author{V. Kumar Murty} 
\thanks{Research of VKM partially supported by an NSERC Discovery Grant.}
\address{Department of Mathematics, University of Toronto, 40 St. George St.,
Toronto, CANADA M5S 2E4 
}
\email{murty@math.toronto.edu}
\author{Vijay M. Patankar}
\address{Indian Statistical Institute c/o Society for Electronic Transactions and Security (SETS), 
MGR Knowledge City, C.I.T. Campus, Taramani, Chennai, INDIA 600113.
}
\email{vijaypatankar@gmail.com}
\date{\today}

\begin{abstract}
We consider Tate cycles on an Abelian variety $A$ defined over a sufficiently large number field $K$ and having  complex multiplication. We show that 
there is an effective bound $C = C(A,K)$ so that 
to check whether a given cohomology class is a Tate class on $A$, it suffices to check the action of 
Frobenius elements at primes $v$ of norm $ \leq C$. 
We also show that for a set of primes  $v$ of $K$ of density $1$, the space of Tate cycles on the special fibre $A_v$ of the N\'eron model of $A$ is isomorphic to the space of Tate cycles on $A$ itself. 
\end{abstract}

\maketitle
 
\section{Introduction and Statement of Results}\label{intro}


Let \( K \) be a 
number field or a finite field. Let \( A \) be an Abelian variety defined over $K$ of dimension \( d \), say. For a prime 
\( \ell \) with \( \ell \notdiv char(K) \),
and \( n \geq 1 \), we have 
the \( \ell \)-adic Tate module
\[
T_\el (A) = \projlim A[ \ell^n] \simeq \Z_\el^{2d} .
\]
In a natural way, it is a \( G_K = Gal ( \Kbar / K ) \)-module. Set also
\[
V_\el (A) = T_\el (A) \otimes_{\Z_\el } \Q_\el . 
\]
Set \( \bar{A} = A \times_K \bar{K} \).
Then, we may identify $V_\ell(A)$ with $H_{1,\ell}(\overline{A})$ and the cohomology of $A$
with the exterior algebra on the dual of $V_\ell(A)$. 
 As usual, for a $G_K$-module $W$, we denote by $W(k)$ the $k$-fold Tate twist of $W$.
\medskip\par
For any field \( F \) with \( K \subseteq F \subseteq \bar{K} \), set 
\begin{eqnarray*}
\TT_\el^k ( A ,  F) & = & H^{2k}_\el ( \bar{A} )(k)^{ Gal(\bar{K}/F )} \\
\TT_\el^{\ast} ( A ) & =  & \bigoplus_k  \sum_{K \subseteq F \subseteq \bar{K} } \TT_\el^k ( A ,  F) 
\end{eqnarray*}
This is the space of Tate cycles on \( A \). It is a finite dimensional \( \Q_\el \) vector space and we may in fact restrict the field $F$ above to be of finite degree over $K$. Thus, for some finite extension \( M \) of \( K \), we have
\[
\TT_\el^{\ast} ( A ) =  \oplus_k H^{2k}_\el ( \bar{A} )(k)^{Gal(\bar{K}/M )}  = \oplus_k \TT_\el^k ( A, M ) 
\]
Again, for \( K \subseteq F \subseteq \bar{K} \), let \( Z^k ( A, F) \) denote the free Abelian group generated by 
algebraic cycles of co-dimension \( k \) on \( A \) 
(modulo homological equivalence)
with a representative defined over \( F \). The 
\( \el \)-adic cycle class map is 
\[
c_{\el,k,F} ~: ~  Z^k ( A, F) \otimes_{\Z} \Q_l \rightarrow \TT_\el^k (A , F) .
\]
\begin{conjecture}\cite{Tate-Conj}(Tate)\label{Tateconjecture1}
The $\ell$-adic cycle class map is surjective. 
\end{conjecture}

This is known for $k=1$ by Tate in the case $K$ is a finite field and by Faltings in the case $K$ is a number field.
This conjecture is not known to be independent of \( \el \). 
\medskip\par
There is a related conjecture involving \( L \)-functions associated to \( A \). 
%

\begin{conjecture}(Tate)\label{Tateconjecture2}
If \( K \) is a number field, \( L ( H_\el^{2k} (\bar{A} ), s ) \) has a meromorphic continuation to \( Re(s) = 1 + k \) and 
\[
\ord_{s=1+k} L ( H_\el^{2k} (\bar{A} ), s ) = - \dim \TT_\el^k ( A, K ) .
\]
If \( K \) is a finite field, 
\[
\ord_{s=k} L ( H_\el^{2k} (\bar{A} ), s ) = - \dim \TT_\el^k ( A, K ) .
\]
\end{conjecture}
\medskip\par


Our aim is to study (in the case \( K \) is a number field) the relationship between \( \TT_\el^\ast ( A) \) and  
\( \TT_\el^{\ast} ( A_v ) \). Here \( v \) is a prime of \( K \) of good reduction for $A$, \( \ell \) is assumed to be 
distinct from the characteristic of the residue field \( k_v \) (of \( K \) at \( v \)) and 
\( A_v \) stands for the reduction of \( A \) modulo \( v \). 
\medskip\par  
At a prime $v$ of good reduction for $A$, we have the natural reduction map, 
\[
A \rightarrow A_v ,
\]
which induces a natural map (under pull-back) 
$$
H_\el^{\ast} ( \overline{A_v} ) \simeq H_\el^{\ast} ( \overline{A} ) .
$$
Let us set
\begin{equation}\label{isom-of-cohomologies-modulo-v}
\iota_v: H_\el^{\ast} ( \overline{A} )\ \longrightarrow\ H_\el^{\ast} ( \overline{A_v} )
\end{equation}
be the inverse of the above isomorphism. 
We may view both sides as $\Q_\ell[G_v]$-modules where \( G_v \subset G_K \) is \emph{the} decomposition group at 
\( v \) (subgroup of \( G_K \) unique up to conjugation). Indeed, the left hand side 
\( H_\el^{\ast} ( \overline{A} ) \) is naturally a \( G_v \)-module under restriction from \( G_K \).
The right hand side is naturally a $\Q_\ell[G_{k_v}]$ module 
where $G_{k_v}$ is the Galois group of the residue field $k_v$ at $v$. Since \( G_{k_v} \) is naturally a quotient of \( G_v 
\), we can consider \( H_\el^{\ast} ( \overline{A_v} ) \) as a  $\Q_\ell[G_v]$-module.  Also, as $v$ is a prime 
of good reduction, the criterion of N\'eron-Ogg-Shafarevich \cite{Serre-Tate} implies that the inertia subgroup at $v$ acts trivially on \( H_\el^{\ast} ( \overline{A} ) \). Thus, the two sides in 
\eqref{isom-of-cohomologies-modulo-v} are isomorphic as $\Q_\ell[G_v]$-modules and so, restriction  
 gives us a map
\begin{equation}\label{Tate-classes-mod-v}
\iota_v ~ : ~ \TT_\el^{\ast} ( A ) \rightarrow \TT_\el^{\ast } ( A_v ) 
\end{equation} 
which is in fact injective.

%
\medskip\par
We begin with the following observation.
A class \( \omega \in H_\el^{2m}( \bar{A})(m) \) for which $\iota_v(\omega)$  is in \( 
\TT_\el^m ( A_v , k_v) \) at almost all places \( v \) of \( K \) (i.e. at all but finitely many places of \( K \)) is in 
\( \TT_\el^m ( A , K) \). 
Indeed, this follows from the Chebotarev density theorem. This observation can be strengthened as follows.

\medskip\par

\begin{prop}
Let \( S \) be a set of places of \( K \) with positive Dirichlet density, say \( \delta > 0 \). A class \( \omega \in 
H_\el^{2m}( \bar{A})(m) \) for which $\iota_v(\omega)$  is in \( \TT_\el^m ( A_v , k_v) \) at all places \( v \) in \( S \) is 
in \( \TT_\el^m ( A ) \). 
\end{prop}

\pf Note that the \( \ell \)-adic representation \( \rho_\ell : G_K \rightarrow Aut( H_\el^{2m}( \bar{A})(m) )  \) is 
unramified at all but finitely many places \( v \) of \( K \). Let \( L \) be the sub-extension of \( \bar{K} \) fixed by
\( Ker(\rho_\ell ) \). Then \( L \) is a Galois extension of \( K \) and is unramified at all except finitely many places of 
\( K \). At a finite prime \( v \), to be a Tate class means that it is fixed by the \emph{Frobenius automorphism} in 
\( Gal ( \overline{k_v} / k_v ) \) given by \( x \mapsto x^{\N v } \), where \( \N v \) is the cardinality of \( k_v \) (also 
the norm of \( v \)). The Frobenius automorphism lifts to \( Frob_v \) (at almost all \( v \)) and thus, by 
\eqref{isom-of-cohomologies-modulo-v} as above and the assumption, \( \omega \) is fixed by 
\( Frob_v \in Gal(L/K ) \). 
Let \( G_S \) be the group generated by \( \{ {\textrm Frob}_v ~| ~ v \in S \} \) in \( Gal (L/K) \) as a topological group. Let \( L^{G_S} \) be the sub-extenstion of \( L \) cut out by \( G_S \). The places of \( S \) split 
completely in \( L^{G_S} \). This implies that $L^{G_S}$ is a finite extension of $K$. Moreover, by the Chebotarev density theorem, \( [ Gal(L/K) : G_S ] = [ L^{G_S} : K ] \leq 1 / \delta \). Thus, it 
follows that if for all \( v \in S \), $\iota_v(\omega)$  is in \( \TT_\el^m ( A_v , k_v) \), then \( \omega \) is in \( \TT_\el^m ( A , L^{G_S}) \), i.e. it is a Tate class.
\qed

\medskip\par
With more work, we will show that in the case of Abelian varieties with complex multiplication, the condition of the
above result can be replaced with a condition at a {\em finite} set of primes. Moroever, this can be done in an effective manner. Here also, the role of split primes is crucial.
\medskip\par

We recall that for each prime $\ell$, Serre and Tate \cite{Serre-Tate} define a conductor of $A$ (in terms of the Artin and the 
Swan character). It is an integer that they prove (\cite{Serre-Tate}, pp. 499-500) is independent of $\ell$. 
\begin{theorem}\label{finitebound}
Let $A$ be an Abelian variety with complex multiplication and let $K$ be a number field over which both $A$ and its
endomorphisms are defined. Denote by $d$ the dimension of $A$ and by $N$ the \emph{conductor} of $A$.  Then, there is an effective bound \( C = C (A, K) \)  such that if 
\[
\omega \in H_{\ellnot}^{2m}( \bar{A} ) (m), \textrm{~for some prime~} \ellnot , 
\] 
is a Tate class considered as an element of \( H_{\ellnot}^{2m}( \bar{A_v} ) (m) \) for all \( v \) of \( K \) 
with \( N_{K/\Q} v \leq C \), then \( \omega \) is Tate, i.e. \( \omega \in \TT_\ellnot^m ( A , K) \). 
\end{theorem}
\medskip\par
Several remarks are in order. 
\medskip\par
\begin{remark} 
{\rm The main point of the above theorem is that the bound \( C \) (defined by \eqref{definitionC}) depends only on the Abelian variety \( A \) (in fact, only on the conductor, the dimension and the discriminant of the field of complex multiplication) and 
the number field \( K \) and does not depend on \( \ellnot \). However, if we fix a prime \( \ell \), it is possible to prove 
(using the Chebotarev density theorem and Nakayama's lemma for finitely generated modules over \( \Z_\ell \)) the existence of  a 
\emph{constant} (that would depend on \( A \), \( K \) and \( \ell \)) so that similar statement as in the theorem above is 
true.}
\end{remark}
\begin{remark} {\rm The proof of Theorem \ref{finitebound} will produce an explicit expression for $C$ (as defined by \eqref{definitionC}) only in terms of the 
degree $n_K$, the discriminant $d_K$ of $K$, the conductor $N$, the dimension $d$ of $A$ and
the discriminant \( d_F \) of the field $F$ of multiplication of $A$. We have not made any effort to find an 
optimal bound. However, we note that with the assumption of the Riemann Hypothesis for Dedekind zeta functions, the estimate 
for $C$ that the proof of Theorem \ref{finitebound} yields can essentially be replaced by $(\log C)^2$.
It would be very interesting to study to what extent it would be possible to get a bound which is uniform in each of the 
parameters $n_K, d_K, N, d_F $ and $d$.}
\end{remark}
\begin{remark} {\rm Tate \cite{Tate} (section 3, pp.76-77) discusses the notion of an ``almost algebraic'' cycle. A Tate cycle 
is said to be \emph{almost algebraic} if all but finitely many of its specializations are algebraic. According to Tate, this 
notion of almost algebraic cycle seems to be part of the folklore.  Tate adds: ``it is mentioned explicitly 
in \cite{Serre-Bourbaki}, 5.2 that K\"unneth components of the ``diagonal'' are almost algebraic (by part (1) of Theorem 2 of 
\cite{Ka-Me})''. Thus these ``K\"unneth components of the diagonal'' are examples of almost algebraic cycles that are not known 
to be algebraic. Tate further mentions a weaker conjecture than Conjecture (1.1): The space of Tate cycles is spanned by almost 
algebraic cycles. }
\end{remark}
\begin{remark}
{\rm Let us assume the Tate conjecture for Abelian varieties over finite fields. Suppose $\omega$ is a class in $H^{2m}_\ell(\overline{A})(m)$ with the property that $\iota_v(\omega)$ is algebraic for each prime $v$ of $K$ of norm
$\leq C$. Then $\iota_v(\omega)$ is algebraic for all but finitely many $v$. Indeed, by the Theorem, $\omega$ is a Tate class,
and hence so is its reduction $\iota_v(\omega)$. Now by our assumption, it follows that $\iota_v(\omega)$  is algebraic. 
Hence $\omega$ is almost algebraic. 
}
\end{remark}

Our next result is about the field of definition of Tate cycles on an Abelian variety over a finite 
field obtained by reduction of Tate cycle on a fixed Abelian variety over a number field.

\begin{theorem}\label{fieldofdefinition}
Let $A$ be an Abelian variety defined over a number field $K$. Then, there is a bound $D = D(A)$ so that for all
$v$ of good reduction, all Tate cycles on $A_v$ are defined over an extension of the residue field $k_v$ of degree $\leq D$. 
\end{theorem}

\begin{remark} {\rm In fact our proof of this theorem can be suitably adapted to show that 
all the Tate cycles on an Abelian variety \( A \) over a number field \( K \) with complex multiplication by \( F \) are defined over a `specific' extension of \( K \) that depends on \( d = \dim (A) \), the conductor of \( A \) and the normal closure of $F$.}
\end{remark}

As is well known (and as implied by \eqref{Tate-classes-mod-v}) the dimension of the space of Tate cycles does not decrease 
under reduction modulo \( v \). Our next result shows that in the CM case, for a set of primes of Dirichlet density $1$, it 
does not increase either. The proof of this theorem uses Thereom (\ref{fieldofdefinition}).

\begin{theorem}\label{CM}
Let \( A \) be an Abelian variety of CM type. Let \( K \) be sufficiently large so that all the Tate cycles on \( A \) and all the endomorphisms of \( A \) are defined over \( K \). Then, for a set of primes \( v \) of \( K \) of Dirichlet density \( 1 \),
\[
\TT_\el^{\ast} ( A ) \simeq \TT_\el^{\ast} ( A_v ) .
\]
In particular, the Tate conjecture for \( A \) implies the Tate conjecture for \( A_v \) for a set of primes \( v \) of 
Dirichlet density \( 1 \). 
\end{theorem}
 
\begin{remark} {\rm The condition that \( K \) be sufficiently large is necessary as the following example illustrates: Let \( E \) be an elliptic curve over \( \Q \) with complex multiplication by an imaginary quadratic field, 
say \( K = \Q ( \sqrt{-d} )\), with \( d \) square-free and positive integer. Let \( A = E \times E \). 
Then, \( Z^1 ( A, K ) \) is generated by \( E \times \{ 0 \} \), \( \{ 0 \} \times E \), the diagonal \( \Delta := \{ ( x, x ) | x \in E \} \) and the graph of complex multiplication 
$$
 \Delta^{\prime} := \{ ( x , [ \sqrt{-d} ] x ) | x \in E \},
 $$
  where \( [ \sqrt{-d} ~] \) denotes the endomorphism of \( E \) whose square is \( [ - d ~] \). 
\medskip\par
We remark that \( E \) has super-singular reduction at a prime \( p \) of \( \Q \) which remains inert in \( K \) and of good reduction for 
\( E \). Let \( A_p \) denote the reduction of \( A \) modulo \( p \) over \( \FF_p \). Then, \( Z^1 (A_p , \bar{\FF_p} ) \) is 
generated by \( E_p \times \{ 0 \}  \),  \( \{ 0 \} \times E_p \), \( \Delta_p \) (the reduction of \( \Delta \) modulo \( p 
\)) , \( \Delta^{\prime}_p \) (the reduction of \( \Delta^{\prime} \) modulo \( p \)), \( \Delta_\phi := \{ ( x, \phi (x) ) | x 
\in E_ p \} \) and \( \Delta_\psi :=  \{ ( x, \psi (x) ) | x \in E_ p \} \), where \( \phi \) and \( \psi \) are the other two 
generators of \( End_{\bar{\FF_p}} (E_p) \). This proves that for such primes \( p \), the rank of \( \TT_\ell^1 (A) \) 
is \( 4 \) and the rank of  \( \TT_\ell^1 (A_p ) \) is 6. Thus, the \emph{reduction} map as in (\ref{Tate-classes-mod-v}) is 
strictly injective for a set of primes \( p \) of \( \Q \) of Dirichlet density \( \frac{1}{2} \)! However, this is not a 
problem as the set of primes of \( \Q \) that are inert in \( K \), has Dirichet density \( 0 \) considered as a set of primes 
of \( K \). }
\end{remark}

\begin{remark} {\rm The theorem is not true for Abelian varieties without complex multiplication as the following example illustrates\footnote{We thank the referee for suggesting us to bring this to the attention of the reader.}: Let \( E \) be an elliptic curve over \( \Q \) without complex multiplication. 
Let \( A = E \times E \) be as above. Then, \( Z^1 ( A , K ) \) is generated by \( E \times \{ 0 \} \), \( \{ 0 \} \times E 
\) and the diagonal \( \Delta \) for any number field \( K \). Let \( v \) be a finite place of \( K \) 
of good reduction for \( E \) and let \( k_v \) be the finite residue field at \( v \). 
Then, \( E \) always acquires the extra endomorphism (the \emph{Frobenius}) and \(  Z^1 ( A_v , k_v ) \) has rank at least \( 4 \). 
This implies that \( \TT_\ell ( A) \) strictly injects into \( \TT_\ell ( A_v ) \) for all but finitely many places of \( K \). }
\end{remark}
\begin{remark}
{\rm The above result is, in some sense, related to our previous work \cite{MurtyPatankar}. There, we studied the problem of when $A_v$ stays simple for a set of primes of  positive density, given that $A$ is simple (or absolutely simple).  If $A_v$ splits, this gives rise to `extra' classes in the N\'eron-Severi group of $A_v$ (and hence also extra Tate cycles on $A_v$).}
\end{remark}
\medskip\par
In section \ref{lambda}, we recall some basic properties of a compatible family of $\lambda$-adic representations.
In section \ref{Shimura-Taniyama}, we recall some basic properties of the theory of complex multiplication.
In section \ref{least}, we develop an analytic estimate that will be crucial in the proofs of our main results.
In section \ref{MainLemma}, we present the Main Lemma and in the following three sections, we give the proofs of Theorems \ref{finitebound}, \ref{fieldofdefinition} and \ref{CM}.

\section{Brief background on compatible family of \( \lambda \)-adic representations}\label{lambda}

We recall (from \cite{Serre-Abelian-l-adic}) some basics about \emph{compatible} families of \( \lambda \)-adic Galois representations. 
\medskip\par
Let \( K \) be a number field. Let \( \Kbar \) be a separable algebraic closure of \( K \). Let \( G_K \) denote the 
Galois group of \( \Kbar \) over \( K \). Let \( E_\lambda \) be a non-archimedean local field (finite extension of 
some \(p \)-adic field). Let \( V \) be a vector space over \( E_\lambda \). Let \( \Aut (V) \) be the general linear 
group of \( V \) with topology induced by that on \( End (V) \): If \( n = \dim (V) \), we have \( \Aut (V) \simeq \GL 
( n, E_\lambda ) \).

\begin{defn}
A \( \lambda \)-adic representation of \( G_K \) \rm{(} or of \( K \) \rm { ) } is a continuous homomorphism \( \rho : 
G_K \rightarrow \Aut (V) \), where \( V \) is vector space over a non-archimedean local field \( E_\lambda \), a 
finite extension of \( \Q_p \).
\end{defn}

For any number field \( F \), by \( \Sigma_F \) we denote the set of finite places of \( F \). 
Recall that a \( \lambda \)-adic representation \( \rho \) is said to be {\em rational} (respectively, {\em integral}) if there exists a finite subset \( S \) of 
\( \Sigma_K \) such that
\begin{enumerate}[(a)]
\item For any place \( v \) in \( \Sigma_K - S \), \( \rho \) is unramified at \( v \).
\item For \( v \notin S \), the coefficients of \( P_{v, \rho} ( T ) \) belong to \( \Q \) \rm{(}respectively, \( \Z \) \rm{)}.
\end{enumerate}

\medskip\par

Let \( E \) be a number field. For each finite prime \( \lambda \) of \( E \), let \( \rho_\lambda \) be a rational \( \lambda \)-adic representation of \( K \).  For any given finite prime \( \lambda \) of \( E \), let us denote by
\( \ell_\lambda \) the prime of \( \Q \) that lies below the prime \( \lambda \) of \( E \). The system
\( \{ \rho_\lambda \} \) is said to be {\em strictly compatible} if there exists a finite subset \( S \) of
\( \Sigma_K \), called the \emph{exceptional} set, such that:
\begin{enumerate}[(a)]
\item For \( S_\lambda := \{ v | ~ v ~ \mathrm{lies~over}~ \ell_\lambda  \} \) and every \( v \notin S \cup S_\lambda \), \( \rho_\lambda \) is unramified at \( v \) and \( P_{v, \rho_\lambda } ( T ) \) has rational integral coefficients.
\item \( P_{v, \rho_\lambda} (T) = P_{v, \rho_{\lambda^{\prime} } } ( T ) \) if
\( v \notin S \cup S_\lambda \cup S_{\lambda^{\prime} } \).
\end{enumerate}

\begin{defn} Let \( \rho  (= \rho_\lambda )\) be a \( \lambda \)-adic representation of \( G_K \). Then, we say that \( \rho \) is pure of weight \( w  \in \Z \) if there is a finite set \( S \) of finite places of \( K \) such that, for each finite place \( v \notin S \cup S_\lambda \), \( \rho ( Frob_v ) \) is unramified, and the eigenvalues of
\( \rho ( Frob_v ) \) are algebraic integers whose complex conjugates have complex absolute value \( q_v^{w/2} \), where \( q_v \) is the cardinality of the residue field of \( K \) at \( v \). We say that a compatible family \( \{ \rho_\lambda \} \) is of weight \( w \in \Z \) if for each finite place \( \lambda \) of \( E \), \( \rho_\lambda \) is of pure weight \( w \).
\end{defn}
\medskip\par

\medskip\par
As usual, we denote by $\chi_\ell$ the cyclotomic character giving the action of $\Gal(\overline{K}/K)$ on $\ell$-power
roots of unity. The $\{ \chi_\ell \}$ form a compatible family of $\ell$-adic representations.

\section{Abelian Varieties with Complex Multiplication and Shimura-Taniyama theorem}\label{Shimura-Taniyama}

We now briefly recall a few basic facts about Abelian varieties of \emph{Complex Multiplication} type (for short, 
Abelian varieties with CM). The main references are \cite{Lang-CM}, \cite{Pohlmann}, \cite{Serre-Tate}, \cite{Shimura}.  

Let \( A \) be an Abelian variety defined over a number field \( K \) of dimension 
\( d := \dim (A) \). Let \( F \) be a number field of degree \( 2 d \). Then, \( A \) is said to have Complex 
Multiplication by \( F \) (or for short, CM by \( F \)) if there exists an embedding
\[
\iota : F \rightarrow End_\Q ( A ) := End_{ \overline{\Q} }  ( A ) \otimes_\Z  \Q  .
\]
It is a fact that such an \( F \) must be a totally imaginary extension of a totally real number field. 
Let \( E \) be the Galois closure of \( F \). To such an Abelian variety, the Shimura-Taniyama theory associates Hecke 
characters \( \psi_{i,\lambda} \) of the following type. For each \( 1 \leq i \leq 2 d \) and for each finite place 
\( \lambda \) of \( E \) (lying over a prime \( \ell \)), we have a family of continuous characters of the Galois group \( G_K := Gal ( \Kbar / K ) \):
\[
\psi_{i , \lambda} : G_K \rightarrow E_\lambda^\times 
\]
with the property that for all \( v \) not dividing \( \ell N \), 
\( \psi_{i,\lambda} ( Frob_v ) \in F^\times \).
Thus, the field generated by \( \psi_{i,\lambda} (Frob_v ) \) for all \( v \) over \( \Q \) is contained in \( E \).  
We have the following well-known result (\cite{Pohlmann}, Theorem \( 2^{\prime} \), page 171):

\begin{theorem}(Shimura-Taniyama) \label{ShTaCMCohoStruct} (With notations and definitions as above) 
Let \( A \) be an Abelian variety defined over a number field with Complex Multiplication \( F \). Let \( H^k_\ell ( \overline{A} ) \) be the \( k \)-th \( \ell \)-adic \'{e}tale cohomology of \( A \). Then,
\[
H^k_\ell ( \overline{A} ) \otimes_{\Q_\ell }  \overline{\Q_\ell} = \wedge^k H^1_\ell ( \overline{A} ) \otimes_{\Q_\ell }  \overline{\Q_\ell}  = \oplus_I H_I ,
\]
where the direct sum is over subsets \( I \subseteq \{ \psi_{1, \lambda}, \ldots , \psi_{2d,\lambda} \} \) of cardinality \( k \) and where each \( H_I \) is a one dimensional \( G_K \) invariant \( \overline{\Q_\ell} \) subspace of \( H^k_\ell ( \overline{A} ) \otimes_{\Q_\ell } \overline{\Q_\ell} \). Further more, for any finite place \( v \) of \( K \) away from \( \ell \) and from places of bad reduction for \( A \), and for all \( x \in H_I \) 
\[
( Frob_v )^{-1} ( x ) = ( \prod_{ \psi_{i, \lambda} \in I } \psi_{i,\lambda} ( \Frob_v ) ) \cdot x 
\]
\end{theorem}
In fact, the Shimura-Taniyama theory implies that for \(  1 \leq i \leq 2d \), \( ( \psi_{i, \lambda} )_\lambda \) 
forms a compatible family of \( 1 \)-dimensional continuous \( \lambda \)-adic representations of \( G_K \), where \( \lambda \) runs over finite places of \( E \). We denote these compatible families by \( \psi_i \). 
Let \( \Psi := \{ \psi_1 , \ldots, \psi_{2d} \} \) and \( \Psi_\lambda := \{ \psi_{1, \lambda} , \ldots, \psi_{2d, \lambda} \} \). 

\begin{theorem}(With notations as above) \label{CMCoho-lambda-adic} 
Let \( \lambda \) be a prime of \( E \) lying over \( \ell \). Then,
\begin{equation}\label{family}
H^{2k}_\ell ( \overline{A} ) \otimes_{\Q_\ell }  E_\lambda =  \oplus_I H_{I, \lambda }  
\end{equation}
where \( I \subseteq \Psi_\lambda \) of size \( k \) and \( H_{I , \lambda} \) is a 1 dimensional \( G_K \) 
invariant \( E_\lambda \)-subspace such that for any finite place 
\( v \) of \( K \) away from \( \ell \) and places of bad reduction for \( A \),
\[
( Frob_v )^{-1}( x ) = ( \prod_{\psi_{i,\lambda} \in I } \psi_{i,\lambda} (Frob_v ) ) \cdot x
\]
for \( x \in H_{I, \lambda} \).
\end{theorem}

We separate this as a corollary:

\begin{cor} \label{lambdacompat}
With notations and definitions as above, \( \{ H_{I, \lambda } \} \) is a strictly-compatible family of \( \lambda \)-adic representations of \( G_K \) as \( \lambda \) varies over finite places of \( E \). 
\end{cor}

\section{The least prime that does not split completely}\label{least}

Let \( L \) and \( K \) be number fields with \( K \subseteq L \). In the proof of our results, it will be necessary to have an 
estimate for the norm of a prime \( v \) of \( K \) that does not split completely in \( L \). Such an estimate is 
given in  \cite{forum}, Theorem 1 assuming the 
Generalized Riemann Hypothesis for the Dedekind zeta function \( \zeta_L (s) \) of \( s \) 
and assuming that \( L/K \) is Galois. In the case that $K = \Q$, an unconditional estimate is given by X. Li in \cite{XLI}, Theorem 1
and also by Vaaler and Voloch in \cite{Vaaler-Voloch}. However, we need an (unconditional) estimate for a general \( K \). Such an estimate is remarked after the proof of Theorem 1 in 
\cite{forum}. What we prove is slightly weaker, though more than sufficient for our purposes. 
\medskip\par
Let us denote by  \( n_L , n_K \)  the degrees of \( L \) and \( K \) over \( \Q \) respectively. Let \( d_L , d_K \) denote the discriminants of \( L / \Q \) and \( K / \Q \), respectively. 
We also note the effective prime number theorem in $K$ as given by Lagarias and Odlyzko 
(see \cite{Lagarias-Odlyzko} and \cite{Serre-IHES}, Th\'eor\`eme 2). This theorem tells us that the number $\pi_K(x)$ of primes of $K$ of norm $ \leq x$ satisfies
\begin{equation}\label{pnt}
|\pi_K(x) -  \Li x + \Li(x^\beta)| \ \leq  c_1x\exp(-c_2\sqrt{\frac{\log x}{ n_K}}))
\end{equation}
provided $x \geq c_3n_K(\log |d_K|)^2$.
Here, $c_1, c_2, c_3  > 0$ are absolute and effective constants and $\beta$ is the possible exceptional zero of the Dedekind zeta function $\zeta_K(s)$ . If it exists, it is real and satisfies
$$
1 - \frac{1}{4\log |d_K|} < \beta < 1.
$$
Let us set
\begin{equation}\label{fofK}
f (K) = 
\begin{cases}
n_K^2 ~~~~~~~~~~~~~~~~~~~~~~~~~~~~~~~~~~~{\rm if} ~ \zeta_K (s) ~ {\rm\ does\ not\ have\ an\ exceptional\ zero} \\
\max(n_K!\log |d_K|, |d_K|^{1/n_K} ) + n_K^2  ~~~~~~~~~~~~{\rm otherwise.}
\end{cases}
\end{equation}
Thus, for example $f(\Q) = 1$.
\begin{theorem}\label{unconditional-forum}
There is an effective and absolute constant $c > 0$ with the following property. Let \( K \) be a number field and let \( L / K \) be a finite non-trivial Galois extension of degree $n$ . Then, there exists
a prime ideal \( \wp \) of \( K \) such that
\begin{itemize}
\item \( \wp \) is of degree 1 over \( \Q \) and unramified in \( L \). \\
\item \( \wp \) does not split completely in \( L \) and 
\begin{equation} \label{uncond-bound}
\N_{K / \Q } ~ \wp ~ < \ \max(55, e^{cf(K)}| d_L |^{5/2(n-1)} )
\end{equation}
\end{itemize}
\end{theorem}
\begin{proof} 
From \cite{Stark}, Lemma 3, we have for $\sigma > 1$,  the inequality
\begin{equation}\label{stark}
-\frac{\zeta^\prime_L}{\zeta_L}(\sigma) < \frac{1}{\sigma} + \frac{1}{\sigma-1} + \half\log\left( \frac{|d_L|}{2^{2r_2}\pi^{n_L}}\right)
+ \frac{r_1}{2}\frac{\Gamma^\prime}{\Gamma}(\sigma/2) + r_2\frac{\Gamma^\prime}{\Gamma}(\sigma).
\end{equation}
On the other hand, we have the Dirichlet series expansion
$$
-\frac{\zeta^\prime_L}{\zeta_L}(\sigma) \ =\  \sum_{\mathfrak p, m} \frac{\log \N\mathfrak p}{(\N\mathfrak p)^{m\sigma}}
$$
where the sum ranges over primes $\mathfrak p$ of $L$ and $1 \leq m \in \Z$. 
If all primes $\wp$ of $K$ which are of degree $1$ over $\Q$  and of norm less than $y$ (say) ramify or split completely in $L$, then we see that with $n = [L:K] = n_L / n_K $, 
$$
-\frac{\zeta^\prime_L}{\zeta_L}(\sigma) \geq n\sum_{\N\wp \leq y}\!\!{}^\prime\  \frac{\log \N\wp}{(\N\wp)^\sigma}
$$
where the prime on the summation indicates that we range over primes $\wp$ of $K$ that are of degree $1$ over  $\Q$ and which are unramified over $L$. 
Thus, 
$$
-\frac{\zeta^\prime_L}{\zeta_L}(\sigma) \geq n\sum_{\N\wp \leq y}\  \frac{\log \N\wp}{(\N\wp)^\sigma}\ -\ nS_1 - nS_2
$$
where
$$
S_1 = \sum_{\N\wp \leq y} \!\!{}^r \ \ \frac{\log \N\wp}{(\N\wp)^\sigma}
$$
and  the sum is taken over primes $\wp$ of $K$ that ramify in $L$, and
$$
S_2  =  \sum_{ \stackrel{ \N\wp \leq y}{\wp {\rm\ of\ degree\ } \geq 2 } } \ \frac{\log \N\wp}{(\N\wp)^\sigma}.
$$
For each prime $\wp$ counted in this sum, let $p$ be the rational prime that it divides. Then $(\log \N\wp)/(\N\wp)^\sigma\ \leq 2(\log p)/p^{2\sigma}$. Moreover, given a rational prime $p$, there are at most $n_K$ primes $\wp$ of $K$ dividing $p$ and so
$$
S_2 \ \leq\ 2 n_K\sum_p \frac{\log p}{p^{2\sigma}}\ \leq \ 2n_K.
$$
We are using here the fact that for the Riemann zeta function, the inequality \eqref{stark} gives 
$$
\sum_p \frac{\log p}{p^{2\sigma}} \ <\ -\frac{\zeta^\prime}{\zeta}(2\sigma) \ <\ 1
$$
for $\sigma > 1$. 
For $S_1$, we have (\cite{forum}, p. 558)
$$
S_1 \ \leq\ \sum_{\wp|d_{L/K}} \log \N\wp \ \leq \frac{2}{n}\log d_L.
$$
Thus, 
\begin{equation}\label{inequality}
-\frac{\zeta^\prime_L}{\zeta_L}(\sigma) \geq n\sum_{\N\wp \leq y}\  \frac{\log \N\wp}{(\N\wp)^\sigma} - 2\log d_L - 2n_L.
\end{equation}
To estimate the sum on the right,  we use \eqref{pnt}.
Let us assume that $y \geq y_0 = c_3n_K(\log |d_K|)^2$.
By partial summation, we have
\begin{equation}\label{partialsummation}
\sum_{\N\wp \leq y}\  \frac{\log \N\wp}{(\N\wp)^\sigma}\ =
\pi_K(y)(\log y)y^{-\sigma}\ +\ \int_1^y \pi_K(t)t^{-1-\sigma}(\sigma\log t + 1) dt.
\end{equation}
The first term on the right is bounded by using the  estimate
$$
\pi_K(x) \ \leq\ n_K\pi(x).
$$
Thus, we see that as $\sigma > 1$, 
$$
\pi_K(y)(\log y)y^{-\sigma} \ \leq\ c_4 n_K.
 $$
Similarly, we see that
$$
\int_1^{y_0}   \pi_K(t)t^{-1-\sigma}(\sigma\log t + 1) dt \ \leq
\ 2c_4n_K\log y_0
$$
The estimate \eqref{pnt} implies that replacing $\pi_K(t)$ with $\Li(x) - \Li(x^\beta)$ in the integral in 
\eqref{partialsummation} results in an error  of at most  
$$
\leq 2\int_{1}^{\log y} \exp(-c_1\sqrt{u/n_K}) \ u du\ \leq  c_5\ n_K^2.
$$
The term coming from the possible exceptional zero contributes to the integral an amount which is easily seen to be
\begin{equation}\label{exceptional-estimate}
<  (\sigma-\beta)^{-1} \ < (1-\beta)^{-1}.
\end{equation}
To estimate this,
we note that by \cite{Stark}, equations (27) and (28), we have
$$
\beta < \max(1 -  (4n_K!\log |d_K|)^{-1}, 1 - (c_6|d_K|^{1/n_K})^{-1}).
$$
This implies that
$$
(1-\beta)^{-1}\ <\ \max(4n_K!\log |d_K|, c_6|d_K|^{1/n_K} ).
$$
Let us denote the right hand side by $c_K$. 
Finally, it remains to estimate
$$
\int_{y_0}^y (\Li t) t^{-1-\sigma}(\sigma\log t + 1) dt.
$$
and this is easily seen to be
$$
\sigma \int_1^y t^{-\sigma}dt  + O(\int_1^y t^{-\sigma}(\log t)^{-1}) dt.
$$
Taking 
\begin{equation}\label{sigma}
\sigma = 1 + (\log y)^{-1}, 
\end{equation}
 this is easily seen to be
$$
\geq\ \log y - c_7\log\log y.
$$
Putting all this together into \eqref{inequality}, we deduce that 
$$
-\frac{\zeta^\prime_L}{\zeta_L}(\sigma) \geq n\log y  - n(c_K + c_5 (\log |d_K|)^2 + c_7\log\log y) - c_8n_L\log y_0 - 2\log d_L - 2n_L.
$$
On the other hand, we can combine this with the upper bound given above \eqref{stark}. As mentioned in \cite{Stark}, p. 142, we have $\Gamma^\prime(\sigma/2) < 0$ and $\Gamma^\prime(\sigma) < 0$ for $1 < \sigma < 5/4$. Moreover, $\Gamma(x) > 0$ for $x > 0$ real. Hence, we deduce that provided $y > e^4 \sim 54.6$, there is an absolute 
and effective constant $c_9>0$ such that
$$
n\log y <  1 + \log y + \frac{5}{2}\log |d_L|  + c_9nf(K)
$$
with $f(K)$ given by \eqref{fofK}.
It follows that there is a $c > 0$ such that
$$
\log y \ \leq\  cf(K) +  \frac{5}{2(n-1)}\log |d_L|.
$$
Thus, if $y > e^{cf(K)}|d_L|^{5/2(n-1)}$, we get a contradiction and this
 proves the result.
\end{proof}




We record here the following estimate of Hensel (see for example, pp. 44-45 of \cite{RMKM}) which we shall also need. We have
\begin{equation}\label{Hensel-Estimate}
\log d_L \leq ( n_L - 1 ) \sum_{p \in P ( L/ \Q) } \log p  + n_L ( \log n_L ) | P ( L / \Q ) | 
\end{equation}
where \( P ( L / \Q ) \) is the set of rational primes \( p \) that ramify in \( L \). In fact, 
when \( L \) is Galois over \( K \), the following stronger estimate holds:
\begin{equation}\label{Hensel-Estimate-Galois}
\log d_L \leq ( n_L - n_K ) \sum_{p \in P ( L/ K) } \log p  + n_L ( \log n_L - \log n_K  ) + \frac{n_L}{n_K}\log d_K.
\end{equation}
\medskip\par

In particular, if \( L / K \) is a Galois extension unramified outside of primes dividing \( M \) then 
\begin{equation}\label{disc-bound}
\log d_L \ll  n_L \log ( M\frac{n_L}{n_K} ) + \frac{n_L}{n_K}\log d_K.
\end{equation}

\section{Proof of the Main Lemma}\label{MainLemma}
\medskip\par
\begin{lem}
Suppose \( \omega \in H^{\ast}_l ( \overline{A} ) \) is a (simultaneous) eigenvector for a conjugacy set \( C \subset G_K \). Then so is every element of \( \Q_l [ G_K ] (\omega ) \), the \( G_K \)-module generated by \( \omega \) inside \( H^{\ast}_l ( \overline{A} ) \).
\end{lem}
\proof  We need to show that for any \( g \in G_K \), \( g ( \omega ) \) is an eigenvector for any \( \sigma \in C \). We have
\begin{eqnarray*}
\sigma ( g ( \omega ) ) & = & g ( (g^{-1} \sigma g ) ( \omega ) ) \\
& = & g ( \tau ( \omega ) ), {\rm for~some~} \tau \in C {\rm ~as~} C {\rm ~is~a~conjugacy~set} \\
& = & g ( \lambda_\tau \cdot \omega ) ~{\rm for~some~} \lambda_\tau \in \Q_\ell ~{\rm by~assumption} \\
& = & \lambda_\tau \cdot g ( \omega ),
\end{eqnarray*}
the last step follows from the \( \Q_\ell \) linearity of the \( G_K \) action.
\qed
\medskip\par
For integers $N, m, d \geq 1$ and a number field \( K \), let us set
\begin{equation}\label{lem-bound}
B = B(N,K,m,d) = e^{f(K)}N^{mn_Kd^2}(f(K) + n_K\log N)^{mn_Kd^2+1}.
\end{equation}
Here 
\begin{equation}
f (K) = 
\begin{cases}
n_K^2 ~~~~~~~~~~~~~~~~~~~~~~~~~~~~~~~~~~~{\rm if} ~ \zeta_K (s) ~ {\rm\ does\ not\ have\ an\ exceptional\ zero} \\
\max(n_K!\log |d_K|, |d_K|^{1/n_K} ) + n_K^2  ~~~~~~~~~~~~{\rm otherwise.}
\end{cases}
\end{equation}
as defined by \eqref{fofK},  and $n_K = [K:\Q]$.
\medskip\par
Let \( G_K \) denote the absolute Galois group of a number field \( K \). 
Let \( E \) be some number field and \( \mathcal{O}_E \), the ring of integers of \( E \). For a finite place \( \lambda \) of \( E \), we denote by \( E_\lambda \), the completion of 
\( E \) at \( \lambda \), and by \( \mathcal{O}_\lambda \), the ring of integers of \( E_\lambda \). Let $\{ M_\lambda \}$ be a 
family of continuous \( \mathcal{O}_\lambda [ G_K ] \)-modules such that 
\( \{ M_\lambda \otimes_{\mathcal{O}_\lambda} E_\lambda \} \) 
is a strictly compatible family of continuous semi-simple 
$ \lambda $-adic integral representations of \( G_K \) 
of weight $2w$, conductor $N$ and dimension \( d \), where \( \lambda \) varies over all finite places of \( E \). 
Let \( S^{\prime} \) be the set of exceptional places of \( K \) for the system  \( \{ M_\lambda \} \).
\begin{lem} {\bf {Main Lemma}} \label{Forum}  There are absolute and effective constants $c_1, c > 1$  so that if for some finite place \( \lambda_0 \) of \( E \), the set 
\[
S  \ = \ \{ Frob_v: ~ \N v \leq c_1 B(Nwd^2\log d_E,K,n_E,d)^c \}
\]
acts as scalars on $ M_{\lambda_0} $, then $G_K$ acts as scalars on $M_{\lambda_0}$ and
hence on any \( M_\lambda \) where \( \lambda \) is a finite place of \( E \).
\end{lem} 
\pf By abuse of notation, below, we will use \( M_\lambda \) to denote \( M_\lambda \) as a module over \( \mathcal{O}_E \) and as well as the associated vector 
space \( M_\lambda \otimes_{\mathcal{O}_E } E_\lambda \) over \( E_\lambda \). 

First suppose that \( Frob_v \) acts as a scalar \( \mu \) on \( M_\lambda \). Then
\( ( \N v )^{wd } = \mu^d \), since the characteristic polynomial of \( Frob_v \) has coefficients in \( \Z \). So in fact, \( 
\mu = \epsilon ( \N v )^w \) with \( \epsilon^d = 1 \). Since \( \{ \rho_\lambda \} \) is a compatible family of rational 
integral representations, the trace \( d \epsilon ( \N v )^w \) of \( Frob_v \), lies in $\Z$ and so \( \epsilon = \pm 1 \). 
 
If \( d = 1 \) then there is nothing to prove, so we may assume without loss that \( d > 1 \). Suppose that $G_K$ does 
not act as scalars on \( M_{\lambda_0} \). By compatibility, \( G_K \) does not act as scalars on the family \
$\{ M_\lambda \}$. We shall show that \( S \) does not act as scalars either. 

Let \( S^{\prime} \) be the exceptional set (of places of \( K \)) for the system \( \{ M_\lambda \} \). For a finite 
place \( \lambda \) of \( E \), let \( \ell_\lambda \) denote the prime of \( \Q \) that lies below \( \lambda \). Let 
\( S_\lambda := \{ v | ~ v ~ \mathrm{lies~over}~ \ell_\lambda  \} \), the set of finite places of \( K \) that divide 
\( \ell_\lambda \).  For \( v \notin S^{\prime} \cup S_\lambda \), let 
\( P_{v, \lambda} (T) := \det ( T - \rho_{\lambda} ( Frob_v ) ) \). By compatibility, \( P_{v, \lambda} (T) \) is independent of \( \lambda \).

Let \( v \notin S^{\prime} \) be the prime with least norm such that \( Frob_v \) does not act as scalars on $M_{\lambda_0}$. (Note that \( Frob_v \) is unique up to conjugation.) Since \( \rho_\lambda \) is a semi-simple representation, this is equivalent to \( P_v (T) \neq ( T - \theta_v )^d \) for any \( \theta_v \in \C \). Thus by compatibility, \( Frob_v \) does not act as a scalar on \( M_\lambda \) for any \( \lambda \).
\medskip\par
Denote the eigenvalues of \( Frob_v \) by  \( \{ \alpha_{i,v} \} \) for \( i = 1 \) to \( d \). Let 
us choose \( \ell \) unramified in \( E \)
so that the distinct eigenvalues of \( Frob_v \) remain distinct modulo \( \ell \). This can be 
done by choosing an \( \ell \) that does not divide the discriminant of \( E \) and 
the norm of the product of the differences of any two distinct eigenvalues of \( Frob_v \). We have
\[
\left| \prod (\alpha_{i,v} - \alpha_{j,v})\right | \ \leq (2(\N v)^w)^{d^2}
\]
where 
the product is over pairs of distinct eigenvalues of \( Frob_v \). Thus, we can find such an $\ell$ satisfying
\begin{equation}\label{boundonell}
\ell\ \ll \log \{  d_E(2(\N v)^w)^{d^2} \}  \ \ll\ wd^2\log \N v + \log d_E,
\end{equation}
where the implied constant is absoulte. 
\medskip\par
Here, we are using the fact that for an integer $m > 1$, there exists a  prime which is $O(\log m)$ that does not divide $m$. Indeed, by the prime number theorem,  the product of all the primes less than $3\log m$ (say) would be larger than $m$. Hence, at least one of these primes does not divide $m$.
\medskip\par
It then follows that for at least one of the places \( \lambda \) of \( E \) that lie over \( \ell \), the distinct eigenvalues of \( Frob_v \) remain distinct modulo \( \lambda \). Thus, $ Frob_v $ does not act as a scalar on $\overline{M_\lambda} \:=\ M_\lambda /\lambda M_\lambda $. 
\medskip\par
Let \( \overline{ \rho_\lambda } \) be the natural map from
\( G_K \) to \( PGL( \overline{ M_\lambda } ) \), the projective general linear group associated to
\( \overline{ M_\lambda } \). By above, the image of \( Frob_v \) to
\( PGL( \overline{ M_\lambda } ) \) is not the identity. By applying
Theorem \ref{unconditional-forum} to this representation, it follows that there exists a prime $v^\prime$ (say) 
for which the image of $ Frob_{v^\prime} $ in \( PGL( \overline{ M_\lambda ) } \) is not the identity. In particular, 
\( v^{\prime} \) does not split completely in the fixed field \( L \) (say) of the kernel of 
\( \overline{ \rho_\lambda } \). We know that \( L \) is unramified outside \( \ell N \) and as 
$$
\log \frac{n_L}{n_K} \ \leq \ [E:\Q]d^2\log \ell
$$
it follows by \eqref{disc-bound} of section \ref{least},  
\begin{eqnarray}\label{Hensel-Applied}
\log d_L & \ll & n_L[E:\Q]d^2 \log N \ell + \frac{n_L}{n_K}\log d_K .
\end{eqnarray}
Then by (\ref{uncond-bound}) of section \ref{least} and \( n = n_L /n_K \geq 2 \), we can choose $v^\prime$ satisfying
\begin{equation}
\log N v^\prime \ll f(K) + \frac{n_K}{n_L} \log d_L \ll f(K) + n_K[E:\Q]d^2\log N \ell .
\end{equation}
By \eqref{boundonell}, we have
\begin{equation}
\log N v^\prime \ll f(K) + n_K[E:\Q]d^2\log (Nwd^2\log \N v +N\log d_E)  .
\end{equation}
Since \(  Frob_{v^\prime} \) does not act as a scalar on $\overline{M_\lambda}$, it does not act as a scalar on 
$M_\lambda$, and hence on \( M_{\lambdanot} \) either, and as $v$ is a prime of least norm with this property, we have an  absolute and effective constant 
$c > 0$ with
\begin{equation} 
\N v\ \leq \ \N{v^\prime}\ \leq\ \left(e^{f(K)}\{N(wd^2\log \N v + \log d_E\}^{n_K[E:\Q]d^2}\right)^c.
\end{equation}
Hence, we have proved that $\N v$ is bounded by
\begin{equation}
 \ll  (e^{f(K)}(Nwd^2\log d_E)^{n_Kn_Ed^2 })^c(f(K) + n_Kn_Ed^2\log (Nwd^2\log d_E))^{cn_Kn_Ed^2+1}.
\end{equation}
Here the implied constant is absolute and effective.
Thus,
$$
\N v \ \leq c_1B(Nwd^2\log d_E,K,n_E,d)^c.
$$
This contradicts our assumption and proves the lemma. 
  \qed
\medskip\par

\section{Proof of Theorem (\ref{finitebound})}\label{Theorem1.1}
\medskip\par
By assumption, $ \omega \in H^{2m}_\ellnot (\overline{A}) $. Let \( M_\ellnot \) be the $\Q_\ellnot [ G_K ] $ 
sub-module generated by \( \omega \) inside \( H^{2m}_\ellnot ( \overline{A} ) \). Let \( \lambdanot \) be a place of 
\( E \) lying over \( \ellnot \). Let \( M_\lambdanot := M_\ellnot \otimes_{\Q_\ellnot} E_\lambdanot \). Then, by 
Theorem (\ref{CMCoho-lambda-adic}), \( M_\lambdanot \) is isomorphic (as \( E_\lambdanot  [ G_K ] \) module) to a sum 
of \( H_{J,\lambdanot} \) for certain subsets \( J \) of \( \Psi \) of size \( 2m \). Let us denote this set of subsets \( J \) of \( \Psi \) by \( \mathcal{J} \). Thus, \begin{equation}  \label{MLambdanot}
M_\lambdanot = M_\ellnot \otimes_{\Q_\ellnot} E_\lambdanot = \oplus_{J \in \mathcal{J} }H_{J,\lambdanot} .
\end{equation}
The right hand side of (\ref{MLambdanot}) can be realised as the \( \lambda_0 \) component of a family of  $\lambda $-adic representations, say \( \{ M_\lambda \} \), as follows: For any finite place \( \lambda \) of \( E \), let
\[
M_\lambda :=  \oplus_{J \in \mathcal{J}} H_{J,\lambda}.
\]
By the Corollary (\ref{lambdacompat}), \( \{ M_\lambda \} \) is a strictly compatible family of semi-simple \( \lambda \)-adic representations of \( G_K \). 
It is easy to see from the definition that the conductor of this family is bounded by the conductor of the family 
$\{H^{2m}_\lambda\}$. Moreover, the conductor of $H^{2m}_\lambda$ can be bounded in terms of $m$ and the conductor $N$ 
of $A$. In particular, we can get  a bound depending on $N$ and the dimension $d$ of $A$ that majorizes the conductor 
of all the $\{H^{2m}_\lambda\}$. 
\medskip\par
We have $n_E \leq (2d)!$ and the discriminant satisfies
$$
\log d_E \ \leq (2d)!\log d_F.
$$
Notice that
$$
w = m \leq d,
$$
and the dimension of the $M_\lambda$ is equal to 
$$
\binom{2d}{2m} \ \leq 2^{2d}.
$$
Thus, if we set
\begin{equation}\label{definitionC}
C(N,d,F,K) = c_1B(2^{4d}(2d+1)!N\log d_F, K, (2d)!, 2^{2d})^c
\end{equation}
where $c,c_1$ and $B$ are as in Lemma \ref{Forum}, then 
applying Lemma \ref{Forum} to
\( \{ M_\lambda \} \), we get that \( G_K \) acts as scalars on \( M_\lambdanot = \Q_\ellnot [ G_K ]( \omega )   
\otimes_{\Q_\ellnot} E_\lambdanot \). This implies that \( M_\lambdanot = \Q_\ellnot \cdot \omega \otimes_{\Q_\ellnot} E_\lambdanot \). Thus, \( 1 = \dim_{E_\lambdanot } M_{\lambdanot} = \dim_{\Q_\ellnot} M_{\ellnot}  \), proving Theorem (\ref{finitebound}).  
\qed

\section{Proof of Theorem \ref{fieldofdefinition}}
Denote by \( P ( A_v, T ) \) the characteristic polynomial of \( Frob_v \). Writing
\[
P ( A_v , T ) = \prod ( 1 - \alpha_i T ) \ \in \Z[T]
\]
we see that the \( \alpha_i \) are algebraic of degree \( \leq \deg P( A_v , T ) = 2 \dim A_v = 2 \dim A \). 

Moreover, for any \( n \geq 1 \), if \( k_{v,n} \) denotes the extension of \( k_v \) of degree \( n \),
\[
P ( A_v / k_{v,n} , T ) = \prod ( 1 - \alpha_i^n T ).
\]
Now,
\[
\dim \TT_\ell^k ( \overline{A_v}, k_{v,n} ) = \# \{ I~: ~ \alpha_I^n = q^{kn} \},
\]
here \( \alpha_I = \prod_{ \alpha_i \in I} \alpha_i \) where  \( I \) runs over subsets of 
\( \{ \alpha_1 , \cdots, \alpha_{2d} \} \) of cardinality \( 2k \). 
The right hand side is equal to
\[
\# \{ I~: ~ \alpha_I = q^k \zeta_n ~\emph{for~some~} \zeta_n \in \mu_n \}
\]
But \( \alpha_I \) is an algebraic number of degree \( \leq b_{2k} (A) \ = \ \dim H^{2k}_\ell ( \overline{A} ) 
\leq \binom{2d}{2k} \). Hence if \( \alpha_I = q^k \zeta_n \) then \( n \) is bounded. In other words, all Tate classes are defined over an extension of \( k_v \) of degree bounded independently of \( v \). 
\qed

\section{Proof of Theorem \ref{CM}}


By definition \(  \TT_\ell^k (A)\) consists of cohomology classes \( x \in
H_\el^{2k} ( \overline{A} ) \) on which \( G_K \) acts by the character
\( \chi_\ell^k \), where \( \chi_\ell \) is the cyclotomic character of \( G_K \) acting on the Tate module 
\( \Z_\ell (1) \). Note that we are assuming that \( K \) is sufficiently large so that all the Tate classes appear 
over \( K \). By Theorems (\ref{ShTaCMCohoStruct}) and (\ref{CMCoho-lambda-adic}), we see that \( G_K \) acts on
the 1-dimensional subspaces \( H_I \) by the Hecke character \( \psi_I
\). Here, 
\[
\psi_I = \prod_{j=1}^{2k} \psi_{i_j} ,
\]
where $I$ is the subset \( \{ i_1, \ldots, i_{2k} \} \) 
of \( \{ 1 , \cdots, 2d \} \).
Thus,
\begin{equation} \label{Tate-global}
\dim  \TT_\el^{k} ( A ) =  \# \{ I ~ | ~ \psi_I = \chi_\ell^k \}.
\end{equation} 
Similiarly, for a finite place \( v \) of \( K \) of good reduction for \( A \), 
\begin{equation}\label{Tate-local}
\dim \TT_l^k ( A_v ) = \# \{ I ~:~ \psi_I ( Frob_v ) = \zeta_n ( \N v )^k \}
\end{equation}
where \( \zeta_n \) is some \( n \)-th root of unity where \( n \leq \binom{2d}{2k} \) is bounded indepedently of \( v \)
as proved in Theorem (\ref{fieldofdefinition}). Indeed, by Theorems (\ref{ShTaCMCohoStruct}) and (\ref{CMCoho-lambda-adic}), we see that \( Frob_v  \) acts on the 1-dimensional vector subspaces \( H_I \) by multiplication by \( \psi_I ( Frob_v ) \). On the other hand, by above and by Theorem (\ref{fieldofdefinition}), \( \TT_\ell^k ( A_v ) \) consists precisely of those cohomology classes 
on which \( Frob_v \) acts by \( \chi_\ell^k (Frob_v ) = \zeta_n ( \N v )^k \), and so \eqref{Tate-local} follows.
 
Thus, for all finite places \( v \) of good reduction, \eqref{Tate-global} and \eqref{Tate-local} implies:
\begin{eqnarray*}
S := \{ I~ :~ \psi_I = \chi_\ell^k \} \subset S_v := \{ I~ : ~ \psi_I ( Frob_v ) = ( \N v )^k \} 
\end{eqnarray*}
Let \[ T := \{ v \ | \ S \subsetneq  S_v \} . \] 
For any \(  v \in T \), there exists some \( J \notin S \) (\( J \) depends on \( v \)) of cardinality \( 2k \) such that \( \psi_J ( Frob_v ) \neq \zeta_n ( \N v)^k \), where \( \zeta_n \) is some \( n \)-th root of unity. 
We can write \( T = \cup_{J \notin S} T_J \), 
\[
T_J  := \{ v \in T \ | \ \psi_J ( Frob_v ) = \zeta_n ( \N v )^k \}  
\]
We want to prove that \( T \) is of density \( 0 \). Suppose not. By above, since there are only finitely many index sets \( J \), there exists a \( J_0 \notin S \) such that \( \delta ( T_{J_0} ) > 0 \). Hence, for all \( v \in T_{J_0} \), \( \psi_{J_0}^n (Frob_v) = 
( \N v )^{kn} = \chi_\ell^{kn} ( Frob_v ) \). Applying Theorem 2, p. 163 of \cite{Rajan-smo-l} to the characters \( \psi_{J_0}^n \) and \( \chi_\ell^{kn} \), we deduce that 
\[
\psi_{J_0}^n = \chi_\ell^{kn} \chi_1
\]
where \( \chi_1 \) is a character of finite order. Hence \( \psi_{J_0}^n = \chi_\ell^{kn} \) when restricted to
\( G_L := \Gal ( \overline{K} / L ) \), where \( L \) is the fixed field cut out by \( Ker \!( \chi_1 )\). Thus, \( \psi := 
\psi_{J_0}^n \cdot \chi_\ell^{-k} \) is a character of \( G_L \) of order \( n \). Thus \( \psi_{J_0} = \chi_\ell^{k} \) when restricted \( Gal ( \overline{K} / L_\psi ) \), where \( L_\psi \) is the number field cut out by \( Ker \! ( \psi ) \). This implies that \( J_0 \) contributes to a \emph{new} Tate cycle on \( A \). However, this is a contradiction, since by assumption all the Tate cycles are defined over \( K \) and that \( J_0 \notin S \). This proves that \( T \) has density \( 0 \) and that proves the result.
\qed

\end{document}